\documentclass[a4paper,12pt]{article}
\pdfoutput=1
\usepackage{fancyhdr}
\usepackage{eurosym}
\usepackage[english]{babel}
\usepackage[utf8]{inputenc}
\usepackage{ mathrsfs }
\usepackage{amsmath}
\usepackage{amssymb}
\usepackage{amsthm}
\usepackage{bbm}
\usepackage{tikz}
\usetikzlibrary{patterns}
\usepackage{enumerate}
\usepackage{graphicx}
\usepackage{enumerate}

\usepackage[backend = biber, sorting = nyt, defernumbers=true, style = numeric]{biblatex}
\usepackage{url}
\usepackage[bmargin = 4.2cm]{geometry}
\usepackage{microtype}
\usepackage{hyperref}
\RequirePackage{doi}

\usetikzlibrary{decorations.pathreplacing}

\DeclareMathOperator{\Ber}{Ber}

\DeclareMathOperator{\Expo}{Expo}

\DeclareMathOperator{\Unif}{Unif}
\DeclareMathOperator{\Geo}{Geo}

\newcommand{\R}{\mathbb{R}}
\newcommand{\N}{\mathbb{N}}
\newcommand{\Z}{\mathbb{Z}}

\theoremstyle{definition}

\newtheorem{remark}{Remark}
\theoremstyle{plain}
\newtheorem{theorem}{Theorem}

\newtheorem{corollary}{Corollary}
\newtheorem{lemma}{Lemma}

\allowdisplaybreaks

\makeatletter
\let\blx@rerun@biber\relax
\makeatother

\hypersetup{
     pdfauthor={Xaver Kriechbaum},
    pdfborder={0 0 0},
   pdftitle = {Subsequential tightness for BRWRE}
}

\title{Subsequential tightness for branching random walk in random environment}
\author{Xaver Kriechbaum}
\date{}

\bibliography{LitDHCP}

\begin{document}
\maketitle
\begin{abstract}
We consider branching random walk in random environment (BRWRE) and prove the existence of deterministic subsequences along which their maximum, centered at its mean, is tight. This partially answers an open question in \cite{CernyDrewik}. The method of proof adapts an argument developed by Dekking and Host for branching random walks with bounded increments. The question of tightness without the need for subsequences remains open.
\end{abstract}
\section{Introduction, model and main result}
We consider branching random walk in (spatial, time independent) random environment and focus on the study of its maximum. From \cite{CometsPopov}, which proves a shape theorem for a BRWRE on $\Z^d$, $d\ge 1$, one can infer that the maximum satisfies a law of large numbers. Further, a functional central limit theorem for the maximum is proven in \cite{CernyDrewik}. The goal of this paper is to prove tightness along a subsequence for the maximum recentered around its quenched mean. This is motivated by, and partially answers, the third open question in \cite{CernyDrewik}. We only consider the case of a single starting particle.

We begin by introducing the model given in \cite{CernyDrewik} in some more detail. Let $(\xi(x))_{x\in\Z}$ be an i.i.d. collection of random variables on a probability space $(\Omega,\mathcal{F},\mathbb{P})$ with $0<\mathrm{ei}:= \mathrm{ess\,inf}\,\xi(0)< \mathrm{ess\,sup}\,\xi(0) =:\mathrm{es}<\infty$. We use $E_{\mathbb{P}}$ to denote the expected value corresponding to $\mathbb{P}$. Given a realization of $\xi$ and an initial condition $x_0\in\Z$ start with one particle at site $x_0$. All particles move independently according to a continuous-time simple random walk with jump rate 1. While at site $x$ a particle splits into two at rate $\xi(x)$ independently of everything else. These particles then evolve independently according to the same mechanism. We write $P_{x}^\xi$ and $E_{x}^\xi$ (the \emph{quenched} law and expectation respectively) for the law of the process conditioned on starting with a single particle at $x$. Alternatively, we write $P^\xi, E^\xi$ and give our random variables a superscript $x$, which we suppress if $x = 0$. In case $\xi(x) = \xi(0)$ for all $x\in\Z$ we use $P_x^{\xi(0)}$, $E_x^{\xi(0)}$ instead of $P_x^{\xi}$, $E_x^\xi$.  We use $\mathbb{P}\otimes P_{x}^\xi$, $\mathbb{P}\otimes P^\xi$ or just $P_x$ or $P$ to denote the annealed law of the process.

Let $N(t)$ denote the set of particles alive at time $t$, for $Y\in N(t)$ we denote by $(Y_s)_{s\in[0,t]}$ the trajectory of the particle and its ancestors up to time $t$; this is called the genealogy of $Y$. We are interested in $M_t := \max_{Y\in N(t)} Y_t$. 

We use the notation $
\N_{L,\eta} := \N/L+\eta$ with $L\in\N$, $\eta\in [0,1/L)$.

The main result of this paper is the following
\begin{theorem}\label{th:Main}
Fix $L\in\N$, $\eta\in [0,1/L)$ and $\delta>0$. Then, there exists a deterministic subsequence $(t_k^{\delta,\eta})_{k\in\N}$ of $\N_{L,\eta}$ with $\limsup_{k\to \infty}t_k^{\delta,\eta}/k\le (1+ \delta)/L$ so that $(M_{t_k^{\delta,\eta}}-E^\xi[M_{t_k^{\delta,\eta}}])_{k\in\N})$ is tight with respect to the annealed measure.
\end{theorem}
The proof of Theorem \ref{th:Main} yields, with minimal changes, the following quenched result.

\begin{theorem}\label{Theo:Quenched}
Fix $L\in\N$, $\eta\in [0,1/L)$ and $\delta>0$. For $\mathbb{P}$-a.e. $\xi$ there exists a subsequence $(t_k^{\delta,\eta}(\xi))_{k\in\N}$ of $\N_{L,\eta}$ with $\limsup_{k\to \infty} t_k^{\delta,\eta}(\xi)/k\le (1+\delta)/L$ so that $(M_{t_k^{\delta,\eta}(\xi)}-E^\xi[M_{t_k^{\delta,\eta}(\xi)}])_{k\in\N}$ is tight with respect to $P^\xi$.
\end{theorem}

\begin{remark}
We do not know whether it is possible to choose a deterministic subsequence $(t_k)_{k\in\N}$ of $\N_{L,\eta}$ such that $(M_{t_k}-E^\xi[M_{t_k}])_{k\in\N}$ is tight for $\mathbb{P}$-a.e. $\xi$. In particular, Theorems \ref{th:Main} and \ref{Theo:Quenched} do not imply each other.
\end{remark}

To prove Theorem \ref{th:Main} we adapt the Dekking-Host argument \cite{DH91}, which we now briefly recall in the classical context of deterministic branching random walk in discrete time, that is when $\xi(x) = 1$ for all $x$, $\mathbb{P}$-a.s. In that case, we have from the branching structure that, with $M_n, M_n'$ two independent copies of $M_n$ and $W$, $W'$ two independent copies of a $\Ber(1/2)$ random variable taking the values $\pm1$,
\begin{equation*}
\label{eq:ClassDH1}
M_{n+1}\stackrel{d}{=}\max(M_n+W,M_n'+W').
\end{equation*}
Taking expectation and using that $\max(a,b) = (a+b)/2+|a-b|/2$, we obtain that
\begin{align*}
\label{eq:ClassDH2}
E[M_{n+1}] &\ge E[\max(M_n,M_n')] = E[M_n]+E[|M_n-M_n'|]/2\\
&\ge E[M_n]+E[|M_n-E[M_n]|]/2
\end{align*}
and therefore
\begin{equation}
\label{eq:ClassDH3}
E[|M_n-E[M_n]|]\le 2(E[M_{n+1}]-E[M_n]).
\end{equation}
Since $M_{n+1}-M_n\le 1$, Dekking and Host conclude, that $E[M_{n+1}]-E[M_n]\le 1$, which then implies the tightness of $(M_n-E[M_n])_{n\in\N}$ using \eqref{eq:ClassDH3}.

The Dekking-Host argument generalizes to continuous time walks in deterministic environment, with asynchronous jumps and branching; we note that in that case, $M_{n+1}-M_n$ is not deterministically bounded. However, $E[M_n]/n$ converges by the subadditive ergodic theorem, and then moving to subsequences using the argument presented in \cite[p. 9]{OZLN}, which originated in \cite{BDZ1}, yields the analogue of Theorem 1. The case of random environments presents a genuine new difficulty, in that information on $\xi$ is embedded in the law of the configuration at time $1$, and in that (quenched) shift invariance is lost. This requires a considerably more involved argument, that we now describe.

We denote the time of the first split by $\tau_s$ and the time of the first move of any particle with $\tau_m$. We then define $\tau :=\tau_s\wedge \tau_m\wedge (1/L)$ and consider $\mathbf{1}_{\{\tau_s <\tau_m\wedge \frac{1}{L}\}}M_{t+\tau}$. As in the Dekking-Host argument, this has the same distribution as the maximum of two copies $M_{t,1}, M_{t,2}$ of $M_t$, which, given the environment, are independent of each other and also of $\tau_s$ and $\tau_m$. We use this setup in subsection 2.1 to derive the inequality
\begin{equation}
E[|M_{t,1}-M_{t,2}|] \le c^{-1}\left(E[M_{t+\frac{1}{L}}-M_t]+E[\mathbf{1}_{\{\tau\neq \tau_s\}}(M_{t,1}-M_{t+\tau})]\right). \label{eq:MainInequality}
\end{equation}
In order to obtain \eqref{eq:MainInequality}, we prove that $E[\mathbf{1}_{\{\tau_s<\tau_m\wedge 1/L\}}|M_{t,1}-M_{t,2}|]\ge c E[|M_{t,1}-M_{t,2}|]$, for which $\mathrm{ei}>0$ is essential.

We then derive bounds for the two summands in \eqref{eq:MainInequality} along suitable, arbitrarily dense, subsequences of $\N_{L,\eta}$ in the subsections 2.2 and 2.3 respectively.

 For the summand $E[M_{t+\frac{1}{L}}-M_t]$ this is analogous to Corollary 1 in \cite[p. 9]{OZLN} and uses only that $\limsup_{t\to \infty} E[M_t]/t<\infty$ (see Lemma \ref{lem:Oben}):

Remember that we write $M_t^x$ for the maximum at time $t$ starting with a single particle in $x$. We will sometimes use random variables as starting position. For the summand $E[\mathbf{1}_{\{\tau\neq \tau_s\}}(M_{t,1}-M_{t+\tau})]$ we use that on $\{\tau = \tau_m\}$ we have that $M_{t+\tau} \stackrel{d}{=} M_t^{S_1}$, $S_1\sim\Unif(\{-1,1\})$, which reduces the matter to bounding $E[\mathbf{1}_{\{\tau = \tau_m\}}(M_{t,1}-M_t^y)]$, $y\in\{\pm 1\}$.

 For this let $\sigma^y$ be the time at which any particle of the process with a single starting particle in $y$ hits 0. We can then use the descendants of the starting particle for $M_{t,1}$ as descendants, after time $\sigma^y$, of the particle which hits 0. This yields a coupling of $M_{t,1}$ and $M_t^y$ for which $
\mathbf{1}_{\{\sigma^y\le t\}}\mathbf{1}_{\{\tau_m = \tau\}}M_t^y\ge \mathbf{1}_{\{\sigma^y\le t\}}\mathbf{1}_{\{\tau_m = \tau\}}M_{t-\sigma^y,1},$
and it mainly remains to control $E[\mathbf{1}_{\{\sigma^y\le t\}}(M_{t,1}-M_{t-\sigma^y,1})]$.

 To do this we use the fact that there exist constants $c,C_1>0$ for which $P^\xi[\sigma^y\ge z]  \le ce^{-C_1 z}$ (see Lemma \ref{lem:sigExpTail}). We also utilize the bound
\begin{align}
&E[\mathbf{1}_{\{\sigma^y\le t\}}(M_{t,1}-M_{t-\sigma^y,1})]\notag\\
&\qquad\le \sum_{k=1}^{\lfloor L\cdot t\rfloor} E[\mathbf{1}_{\left\{\sigma^y\in \left[\frac{k-1}{L},\frac{k}{L}\right]\right\}}(M_{t,1}-M_{t-\frac{k}{L},1})]+E[\mathbf{1}_{\{\sigma^y\in [t-\eta,t]\}}M_{t,1}].\label{eq:UpperBoundSec23}
\end{align}
Because $\sigma^y$ has exponential tails, it suffices to find subsequences along which $E[M_{t,1}-M_{t-j/L,1}]$, $j\in\{1,\dots, \lfloor L\cdot t\rfloor\}$, are bounded by $c\cdot e^{c'(j-1)}$ with $c,c'$ constants, which are specified below. We first do this separately for each fixed $j$ and get in Corollary \ref{cor:FixjDenseSub} that we can achieve such a bound along arbitrarily dense subsequences. We then intersect these subsequences to get a, arbitrarily dense, subsequence along which $E[\mathbf{1}_{\{\tau\neq \tau_s\}}(M_{t,1}-M_{t+\tau})]$ is bounded (see Lemma \ref{lem:bounded}, Lemma \ref{lem:dense} and Corollary \ref{cor:Unten}). One important observation for this argument is that for fixed $t\in \N_{L,\eta}$, we only need $E[M_{t,1}-M_{t-j/L,1}]$ to be controlled for $j \le c\log(t)$, with $c$ a suitable constant specified below. The reason for this is that $\sigma^y$ has exponential tails and $E[M_{t,1}]$ grows at most linearly. This implies that in \eqref{eq:UpperBoundSec23} for all summands with $k\ge c\log t$ we get a good enough upper bound even if we ignore the $-M_{t-\frac{k}{L},1}$ term.

In subsection 2.4 we combine Lemma \ref{lem:Oben} and Corollary \ref{cor:Unten} to prove Theorem \ref{th:Main}, by intersecting suitably dense subsequences obtained in the aforementioned lemmata.

\noindent
{\bf Acknowledgements}  This project has received funding from the European Research Council (ERC) under the European Union's Horizon 2020 research and innovation programme (grant agreement No. 692452).

Thanks to Ofer Zeitouni for suggesting the problem and for many useful discussions.
\section{Details}
\subsection{Deriving Inequality \eqref{eq:MainInequality}}
 Let $\tau_m$ be the time of the first movement of any particle. Furthermore let $\tau_s$ be the the time of the first split, $\tau_s := \inf\{t\in\R_{\ge0} : |N(t)| = 2\}$. Both $\tau_s$ and $\tau_m$ are stopping times with respect to the filtration generated by $(|N(t)|,(Y_t^{(v)})_{v\in N(t)})_{t\ge0}$.

Let $L\in\N$ be arbitrary but fixed and set $\tau := \tau_s\wedge\tau_m\wedge(1/L)$. Then $\tau$ is also a stopping time with respect to that filtration.

 For $t\in \R_{\ge0}$ we have by definition that $\mathbf{1}_{\{\tau_s<\tau_m\wedge \frac{1}{L}\}}M_{t+\tau} \stackrel{d}{=} \mathbf{1}_{\{\tau_s<\tau_m\wedge \frac{1}{L}\}}\max_{k=1}^2 M_{t,k}$, 
where $M_{t,1}$, $M_{t,2}$ are copies of $M_t$, which are independent of each other and of $\tau_s$ and $\tau_m$ given the environment. Taking expectation and using that $a\vee b = (a+b +|a-b|)/2$ this yields that
\begin{align*}
E[\mathbf{1}_{\{\tau_s<\tau_m\wedge \frac{1}{L}\}}M_{t+\tau}] &=\frac{1}{2}E\left[\mathbf{1}_{\{\tau_s<\tau_m\wedge \frac{1}{L}\}}(M_{t,1}+M_{t,2}+|M_{t,1}-M_{t,2}|)\right]\\
&=E[\mathbf{1}_{\{\tau_s<\tau_m\wedge \frac{1}{L}\}}M_{t,1}]+\frac{1}{2}E[\mathbf{1}_{\{\tau_s<\tau_m\wedge \frac{1}{L}\}}|M_{t,1}-M_{t,2}|].
\end{align*}
By reordering the terms we conclude that
\begin{align}
E[\mathbf{1}_{\{\tau_s<\tau_m\wedge \frac{1}{L}\}}(M_{t+\tau}-M_{t,1})] \ge \frac{1}{2}E[\mathbf{1}_{\{\tau_s<\tau_m\wedge \frac{1}{L}\}}|M_{t,1}-M_{t,2}|]. \label{eq:FVerMainINeq}
\end{align}
Since given the environment $\mathbf{1}_{\{\tau_s<\tau_m\wedge \frac{1}{L}\}}$ is independent of $M_{t,1}$ and $M_{t,2}$ it is also independent of $|M_{t,1}-M_{t,2}|$ given the environment and we have that
\begin{align}
E[\mathbf{1}_{\{\tau_s<\tau_m\wedge \frac{1}{L}\}}|M_{t,1}-M_{t,2}|] = E_{\mathbb{P}}\left[P^\xi\left[\tau_s<\tau_m\wedge \frac{1}{L}\right]E^\xi\left[|M_{t,1}-M_{t,2}|\right]\right].\label{eq:Drei}
\end{align}
\begin{lemma}\label{lem:cx0}
We have for all $\xi$ that $P^\xi[\tau_s<\tau_m\wedge \frac{1}{L}] = \frac{(1-e^{-\frac{1}{L}(\xi(0)+1)})\xi(0)}{\xi(0)+1} =: c_{\xi(0),L}$.
\end{lemma}
\begin{proof}
Given $\xi$ we have that $\min\{\tau_s,\tau_m\}\sim \Expo(\xi(0)+1)$. Thus
\begin{align*}
P^\xi\left[\tau_s<\tau_m\wedge \frac{1}{L}\right] &= P^\xi\left[\tau_s\wedge\tau_m<\frac{1}{L}\right]\cdot P^\xi[\tau_s<\tau_m]\\
&= P^\xi\left[\tau_s\wedge \tau_m<\frac{1}{L}\right]\cdot \frac{\xi(0)}{\xi(0)+1}\\
&= \frac{(1-e^{-\frac{1}{L}(\xi(0)+1)})\xi(0)}{\xi(0)+1}.\qedhere
\end{align*}
\end{proof}
Since $c_{\xi(0),L}$ is increasing in $\xi(0)$ and strictly positive for $\xi(0)>0$ Lemma \ref{lem:cx0} and \eqref{eq:Drei} imply that $
E[\mathbf{1}_{\{\tau_s<\tau_m\wedge \frac{1}{L}\}}|M_{t,1}-M_{t,2}|]\ge c_{\mathrm{ei},L}E[|M_{t,1}-M_{t,2}|].$ 

Together with \eqref{eq:FVerMainINeq} this implies that
\begin{align}
 E[|M_{t,1}-M_{t,2}|] &\le c_{\mathrm{ei},L}^{-1}E[\mathbf{1}_{\{\tau_s<\tau_m\wedge \frac{1}{L}\}}(M_{t+\tau}-M_{t,1})]\notag\\
 &= c_{\mathrm{ei},L}^{-1}\left(E[M_{t+\tau}-M_t]+E[\mathbf{1}_{\{\tau\neq \tau_s\}}(M_{t,1}-M_{t+\tau})]\right).\label{eq:Eins}
\end{align}
Next we aim to simplify this upper bound by replacing $E[M_{t+\tau}-M_t]$ with $E[M_{t+\frac{1}{L}}-M_t]$. For this we prove that $E^\xi[M_t]$ is increasing in $t$.
\begin{lemma}\label{lem:Monoton}
The expression $E^\xi[M_t]$ is monotonically increasing in $t$ for all $\xi$. In particular $E[M_{t+\tau}-M_t]\le E[M_{t+\frac{1}{L}}-M_t]$ and $E[M_t]$ is monotonically increasing in $t$.
\end{lemma}
\begin{proof}
Let $s\ge 0$ and $V\in N_t$ with $V_t = M_t$. Define $V^0 := V$. Given $V^k$ define $V^{k+1}$ as follows. If the particle $V^k$ splits before time $t+s$ choose one of its direct descendants uniformly at random independently of everything else as $V^{k+1}$. Iterate this process, until $V^k$ doesn't split before time $t+s$, which will happen almost surely. We then have that $V^k_t = M_t$ and $M_{t+s}\ge V^k_{t+s}$, which implies that $M_{t+s}-M_t\ge V^k_{t+s}-V^k_t =: \Delta_s$. Since we have chosen the descendants uniformly at random independently of their displacement, $(\Delta_r)_{r\ge0}$ is a time-continuous simple random walk. This implies that $E^\xi[\Delta_s] = 0$ for all $\xi$. This in turn yields that $
E^\xi[M_{t+s}-M_t] \ge E^\xi[\Delta_s]= 0.$
 Thus we conclude that $E^\xi[M_t]$ is monotonically increasing.

The second statement follows, since $\tau\le \frac{1}{L}$, which implies
\begin{align*}
E[M_{t+\tau}-M_t] &= E_{\mathbb{P}}[E^\xi[M_{t+\tau}-M_t]]\le E_{\mathbb{P}}[E^\xi[M_{t+\frac{1}{L}}-M_t]]= E[M_{t+\frac{1}{L}}-M_t].
\end{align*}

Finally, the monotonicity of $E[M_t]$ in $t$ follows by using $E[M_t]= E_{\mathbb{P}}[E^\xi[M_t]]$ and the monotonicity of $E^\xi[M_t]$ in $t$.
\end{proof}

Using Lemma \ref{lem:Monoton} the inequality \eqref{eq:Eins} can be rewritten as
\begin{equation}
E[|M_{t,1}-M_{t,2}|]\le c_{\mathrm{ei},L}^{-1}\left(E[M_{t+\frac{1}{L}}-M_t]+E[\mathbf{1}_{\{\tau\neq \tau_s\}}(M_{t,1}-M_{t+\tau})]\right).\label{eq:RewriteEins}
\end{equation}
This proves \eqref{eq:MainInequality}.

We will handle the two summands in \eqref{eq:MainInequality} separately and find arbitrarily dense subsequences of $\N_{L,\eta}$, $\eta\in[0,1/L)$, along which the summands are bounded. By intersecting the subsequences we will be able to conclude the proof of Theorem \ref{th:Main}.

\subsection{On $E[M_{t+\frac{1}{L}}-M_t]$}
Before we can proceed we need to establish that there exists an $x^\ast\in\R_{\ge0}$ such that $\limsup_{t\to\infty} E[M_t]/t\le x^\ast$.
\begin{lemma}\label{lem:limsup}
There exists an $x^\ast\in\R$ such that $\limsup_{t\to\infty} E[M_t]/t\le x^\ast$.
\end{lemma}
\begin{proof}
By a coupling argument we know $E^\xi[M_t]\le E^{\mathrm{es}}[M_t]$ and it suffices to prove that $
\limsup_{t\to\infty} E^{\mathrm{es}}[M_t]/t<\infty.$

Since the branching rates are constant this is a known result for branching random walks, compare Exercise 1 in \cite[p. 8]{OZLN}.\qedhere
\end{proof}
\begin{lemma}\label{lem:Oben}
Fix $\delta>0$ and $\eta\in \left[0,1/L\right)$. Then, there exists a deterministic subsequence $(t_j^{\delta,\eta})_{j\ge1}$ of $\N_{L,\eta}$ so that $\limsup_{j\to \infty} t_j^{\delta,\eta}/j \le (1+\delta)/L$ and $(E[M_{t_j^{\delta,\eta}+\frac{1}{L}}-M_{t_j^{\delta,\eta}}])_{j\ge1}$ is bounded.
\end{lemma}
\begin{proof}
By Lemma \ref{lem:limsup} there exists an $x^\ast\in\R$ such that $\limsup_{t\to\infty}M_t/t\le x^\ast$. Now fix $\delta\in (0,1)$. Define $t_0^{\delta,\eta} :=0$ and
\[t_{j+1}^{\delta,\eta} := \inf\left\{t\in \N_{L,\eta}: t> t_j^{\delta,\eta}, E[(M_{t+\frac{1}{L}}-M_t)]\le \frac{2x^\ast}{L \delta}\right\}.\]

We have that $t_{j+1}^{\delta,\eta} <\infty$, since else we would have that $E[M_{t_j^{\delta,\eta}+(k+1)/L}-M_{t_j^{\delta,\eta}+k/L}]\ge \frac{2x^\ast}{\delta L}$ for all $k\in\N$ and thus that
\begin{align*}
E[M_{t_j^{\delta,\eta}+k/L}] &= \sum_{ n=0}^{k-1} E[M_{t_j^{\delta,\eta}+(n+1)/L}-M_{t_j^{\delta,\eta}+n/L}]+E[M_{t_j^{\delta,\eta}}]\ge \frac{2x^\ast k}{\delta L}+E[M_{t_j^{\delta,\eta}}],
\end{align*}
which implies that $\limsup_{t\to\infty} E[M_t]/t\ge 2x^\ast/\delta> x^\ast$, which contradicts the choice of $x^\ast$.

By definition we have that $E[M_{t_j^{\delta,\eta}+\frac{1}{L}}-M_{t_j^{\delta,\eta}}]\le  2x^\ast/(\delta L)$ for all $j\in\N$ and we are left with proving that $\limsup_{j\to \infty} t_j^{\delta,\eta}/j \le (1+\delta)/L$.

For this purpose, let $
 K_n := |\{ t\in \N_{L,\eta} : t<n/L+\eta, t\not\in\{t_j^{\delta,\eta}\}|$.
By Lemma \ref{lem:Monoton} we have that $E[M_{t+1/L}-M_t]\ge 0$ for all $t$. Thus by definition of $K_n$ we have that
\begin{equation}
E[M_{\eta+n/L}] \ge 2K_{n}x^\ast/(\delta L). \label{eq:Zwischen}
\end{equation} Lemma \ref{lem:limsup} gives that
\[
\limsup_{n\to \infty} \frac{E[M_{\eta+n/L}]}{\eta+n/L} \le x^\ast,
\] plugging this into \eqref{eq:Zwischen} gives that 
\[
\limsup_{n\to \infty} \frac{K_n}{n}\le  \limsup_{n\to \infty} \frac{\delta}{2x^\ast}\frac{E[M_{\eta+n/L}]}{n/L}\le \frac{\delta}{2}.
\] This implies that
 \[
 \liminf_{n\to\infty} \frac{|\{t\in \N_{L,\eta} : t<n/L+\eta, t\in \{t_j^{\delta,\eta}\}_{j\in\N}|\}}{n}\ge \left(1-\frac{\delta}{2}\right)
 \] which in turn yields that $
 \limsup_{n\to \infty}t_{\left\lceil \left(1-\frac{\delta}{2}\right)n\right\rceil}^{\delta,\eta}/(n/L)\le 1
$ and thus that
\[
\limsup_{n\to \infty} \frac{t_n^{\delta,\eta}}{n}\le L^{-1}\left(1-\frac{\delta}{2}\right)^{-1}\ \stackrel{\delta\in (0,1)}{\le}\  \frac{1+\delta}{L}.\qedhere
\]
\end{proof}
\subsection{On $E[\mathbf{1}_{\{\tau\neq \tau_s\}}(M_{t,1}-M_{t+\tau})]$}By definition we have that $
\mathbf{1}_{\{\tau\neq \tau_s\}} = \mathbf{1}_{\{\tau = 1/L\}}+\mathbf{1}_{\{\tau = \tau_m\}}.$

 On $\{\tau = 1/L\}$ we have that $M_{t+\tau} = M_t'$ with $M_t'$ independent of $\mathbf{1}_{\{\tau = 1/L\}}$ and distributed like $M_t$ and thus
 \[E[\mathbf{1}_{\{\tau = 1/L\}}(M_{t,1}-M_{t+\tau})] = E[\mathbf{1}_{\{\tau = 1/L\}}(M_{t,1}-M_t')] = 0.\]

  Remember that we write $M_t^x$ for the maximum starting with a single particle at location $x$ and allow $x$ to be a random variable. Using this notation we have that
\[
E[\mathbf{1}_{\{\tau = \tau_m\}}(M_{t,1}-M_{t+\tau})] =  E[\mathbf{1}_{\{\tau = \tau_m\}}(M_{t,1}-M_t^{S_1})]
\]
with $M_t^{S_1}$ independent of $\mathbf{1}_{\{\tau =\tau_m\}}$ and $S_1\sim\Unif(\{-1,1\})$ independent of everything. This implies that
\begin{align}
E[\mathbf{1}_{\{\tau = \tau_m\}}(M_{t,1}-M_{t+\tau})] = \frac{E[\mathbf{1}_{\{\tau = \tau_m\}}(M_{t,1}-M_t^{1})]+E[\mathbf{1}_{\{\tau = \tau_m\}}(M_{t,1}-M_t^{-1})]}{2}.  \label{eq:Splitwrtfjump}
\end{align}

Let $y\in\{\pm1\}$. Let $\sigma^y := \inf\{t\ge 0 : \exists V \in N^y(t)\ \text{with}\  V_t^y = 0\}$. Then we can couple $M_{t,1}$ and $M_t^y$ in a way such that
\begin{align*}
\mathbf{1}_{\{\sigma^y\le t\}}\mathbf{1}_{\{\tau_m = \tau\}} M_t^y \ge \mathbf{1}_{\{\sigma^y\le t\}}\mathbf{1}_{\{\tau_m = \tau\}}M_{t-\sigma^y,1}.
\end{align*}
This implies that
\begin{align}
&E[\mathbf{1}_{\{\tau = \tau_m\}}(M_{t,1}-M_t^y)]\notag\\
&\phantom{E}\le E[\mathbf{1}_{\{\tau = \tau_m\}}\mathbf{1}_{\{\sigma^y\le t\}}(M_{t,1}-M_{t-\sigma^y,1})]+E[\mathbf{1}_{\{\tau = \tau_m\}}\mathbf{1}_{\{\sigma^y> t\}}(M_{t,1}-M_t^y)]\notag\\
&\phantom{E}\le E[\mathbf{1}_{\{\sigma^y\le t\}}(M_{t,1}-M_{t-\sigma^y,1})]+E[\mathbf{1}_{\{\tau = \tau_m\}}\mathbf{1}_{\{\sigma^y> t\}}(M_{t,1}-M_t^y)], \label{eq:Decomp}
\end{align}
where the last inequality uses that $E^\xi[\mathbf{1}_{\{\sigma^y\le t\}}(M_{t,1}-M_{t-\sigma^y,1})]\ge 0$, which follows from the monotonicity of $E^\xi[\mathbf{1}_{\{\sigma^y\le t\}}M_{s,1}] = P^\xi[\sigma^y\le t]E^\xi[M_{s,1}]$ in $s$, see Lemma \ref{lem:Monoton}.

 We will handle the two summands in \eqref{eq:Decomp} separately, starting with the second one. However, for both summands, we will need a bound for the tail of $\sigma^y$ which the next lemma provides.

\begin{lemma}\label{lem:sigExpTail}
There exist  $c,C_1>0$ so that $P^{\xi}[\sigma^y\ge z]\le ce^{-C_1z}$ for all $z\ge 0$ and $\mathbb{P}$-a.e. $\xi$.
\end{lemma}
 \begin{proof}

 By coupling we have that $P^\xi[\sigma^y \ge z]\le P^{\mathrm{ei}}[\sigma^y\ge z]$ for $\mathbb{P}$-a.e. $\xi$. For $\tau_y := \inf\{t\ge 0 : \exists V\in N(t) \ \text{with}\  V_t = -y\}$ one has that $P^{\mathrm{ei}}[\sigma^y\ge z] = P^{\mathrm{ei}}[\tau_y \ge z]$. Furthermore, by symmetry we have that $P^\mathrm{ei}[\tau_y\ge z] = P^\mathrm{ei}[\tau_{-1}\ge z]$. By definition of $\tau_y$ we have that
\[
P^\mathrm{ei}[\tau_{-1}\ge z] \le P^{\mathrm{ei}}[M_z \le 1].
\]
Let $\varepsilon>0$. We know that positive constants $c^\ast$, $c'$ exist for which
\begin{align*}
P^{\mathrm{ei}}\left[\min_{Y\in N(\varepsilon z)} Y_{\varepsilon z}\le -c^\ast\varepsilon z\right] \le e^{-c'\varepsilon z},
\end{align*}
compare the upper bound derived in \cite[p. 5]{OZLN}, applied to $\max_{Y\in N(\varepsilon z)}(-Y_{\varepsilon z})$.

With respect to $P^{\mathrm{ei}}$ we have that $(|N(t)|)_{t\ge0}$ is a birth-process with birth rate $\mathrm{ei}$. This implies that $|N(t)|\sim \Geo(e^{-\mathrm{ei}\cdot t})$ see Example 6.8 in \cite[p. 385/386]{IntrPM}. Thus we have that
\[P^{\mathrm{ei}}\left[|N({\varepsilon z})| > e^{\mathrm{ei}\frac{\varepsilon z}{2}}\right]\ge 1-e^{-\mathrm{ei}\frac{\varepsilon z}{2}}.\]

Finally, we know that there exists a $p_\varepsilon>0$, such that
\begin{align*}
P^{\mathrm{ei}}\left[M_{(1-\varepsilon) z} \ge \frac{c^\ast(1-\varepsilon)z}{2}\right]\ge p_\varepsilon.
\end{align*}

Now choose $\varepsilon := 1/12$ then for $z\ge 3/c^\ast$ we have that $
c^\ast(1-\varepsilon)z/2-c^\ast\varepsilon z \ge 1.$ 
This implies that by independence of the particles starting at time $\varepsilon z$
\begin{align*}
P^{\mathrm{ei}}[M_z\le 1]&\le  P^{\mathrm{ei}}[\min_{Y\in N(\varepsilon t)} Y_z\le -c^\ast \varepsilon z]+P^{\mathrm{ei}}\left[N(\varepsilon z)\le e^{\mathrm{ei}\frac{\varepsilon z}{2}}\right]+(1-p_\varepsilon)^{e^{\mathrm{ei}\frac{\varepsilon z}{2}}}\\
&\le e^{-c'\varepsilon z}+e^{-\mathrm{ei}\frac{\varepsilon z}{2}}+(1-p_\varepsilon)^{e^{\mathrm{ei}\frac{\varepsilon z}{2}}},
\end{align*}
for $z\ge 3/c^\ast$. This suffices to conclude Lemma \ref{lem:sigExpTail}
\end{proof}
Armed with Lemma \ref{lem:sigExpTail} we can handle the second summand in \eqref{eq:Decomp}.
\begin{lemma}\label{lem:BoundSideterm}
The sequence $(E[\mathbf{1}_{\{\tau = \tau_m\}} \mathbf{1}_{\{\sigma^y>k/L+\eta\}}(M_{k/L+\eta,1}-M_{k/L+\eta}^y)])_{k\in\N}$ is bounded.
\end{lemma}
\begin{proof}
We have that $M_{t,1}$ is independent of $\mathbf{1}_{\{\tau = \tau_m\}} \mathbf{1}_{\{\sigma^y>t\}}$. Additionally because of $\limsup_{t\to \infty}E^\xi[M_{t,1}]/t\le x^\ast$ by Lemma \ref{lem:limsup}, there exists a $c^\ast\in\R$ so that $E^\xi[M_{t,1}]\le c^\ast t$ for all $t\in \N_{L,\eta}$ and $\mathbb{P}$-a.e. $\xi$. Combining these yields that
\begin{align*}
E[\mathbf{1}_{\{\tau = \tau_m\}} \mathbf{1}_{\{\sigma^y>t\}}M_{t,1}] &= E_{\mathbb{P}}[P^\xi[\tau = \tau_m, \sigma^y>t]E^\xi[M_{t,1}]]\\
&\le E_{\mathbb{P}}[P^\xi[\tau = \tau_m,\sigma^y>t]c^\ast\cdot t]\\
&\le c^\ast\cdot tE_{\mathbb{P}}[P^\xi[\sigma^y>t]]\\
&\le c\cdot e^{-C_1t}\cdot c^\ast\cdot t.
\end{align*}
This converges to $0$ for $t\to\infty$, and in particular is bounded by a constant for $t\in \N_{L,\eta}$.

 Now handle $-E[\mathbf{1}_{\{\tau = \tau_m\}}\mathbf{1}_{\{\sigma^y>t\}} M_t^y]$. We have, using Cauchy-Schwarz in the last inequality, that
 \begin{align*}
 -E[\mathbf{1}_{\{\tau = \tau_m\}}\mathbf{1}_{\{\sigma^y>t\}} M_t^y] &\le E[\mathbf{1}_{\{\tau = \tau_m\}}\mathbf{1}_{\{\sigma^y>t\}} |M_t^y|]\\
 &\le E[\mathbf{1}_{\{\sigma^y>t\}}|M_t^y|]\\
 &\le P[\sigma^y>t]^\frac{1}{2}\cdot\sqrt{E[(M_t^y)^2]}.
 \end{align*}
 We have $E[(M_t^y)^2]\le E^{\mathrm{es}}[(M_t^y)^2]$ by coupling and thus know that there exists a $c^\ast\ge 0$, such that $\lim\limits_{t\to \infty} E[(M_t^y)^2]/t^2 \le c^\ast$. Since $P[\sigma^y>t]^\frac{1}{2}\le ce^{-\frac{C_1t}{2}}$ this does imply that
 \[
 \limsup\limits_{t\to \infty} (-E[\mathbf{1}_{\{\tau = \tau_m\}}\mathbf{1}_{\{\sigma^y>t\}}M_t^y]) \le  0,
 \]
 which in turn gives that the expression in the statement of Lemma \ref{lem:BoundSideterm} is bounded.
\end{proof}
Now we proceed with the first summand in \eqref{eq:Decomp}. For this fix $\eta\in \left[0,\frac{1}{L}\right)$. We have for $t\in \N_{L,\eta}$ that
\begin{align}
E[\mathbf{1}_{\{\sigma^y\le t\}}(M_{t,1}-M_{t-\sigma^y,1})] &\le \sum_{k=1}^{\lfloor L\cdot t\rfloor} E[\mathbf{1}_{\left\{\sigma^y\in \left[\frac{k-1}{L},\frac{k}{L}\right]\right\}}(M_{t,1}-M_{t-\sigma^y,1})]\notag\\
&\qquad+E[\mathbf{1}_{\{\sigma^y\in [t-\eta,t]\}}(M_{t,1}-M_{t-\sigma^y,1})]\notag\\
&\le \sum_{k=1}^{\lfloor L\cdot t\rfloor} E[\mathbf{1}_{\left\{\sigma^y\in \left[\frac{k-1}{L},\frac{k}{L}\right]\right\}}(M_{t,1}-M_{t-\frac{k}{L},1})]\notag\\
&\qquad+E[\mathbf{1}_{\{\sigma^y\in [t-\eta,t]\}}M_{t,1}]\notag\\
&= \sum_{k=1}^{\lfloor L\cdot t\rfloor} E_{\mathbb{P}}\left[P^\xi\left[\sigma^y\in\left[\frac{k-1}{L},\frac{k}{L}\right]\right]E^\xi[M_{t,1}-M_{t-\frac{k}{L},1}]\right] \notag\\
&\qquad+E_{\mathbb{P}}[E^\xi[\sigma^y\in [t-\eta,t]]E^\xi[M_{t,1}] \label{eq:FixEins},
\end{align}
where in the second inequality we use that $E^\xi[\mathbf{1}_{\{\sigma^y\in [(k-1)/L,k/L]\}}M_{s,1}] = P^\xi[\sigma^y\in [(k-1)/L,k/L]E^\xi[M_{s,1}]$ is monotonously increasing in $s$ by Lemma \ref{lem:Monoton}.

Lemma \ref{lem:sigExpTail} and \eqref{eq:FixEins} imply that
\begin{align}
E[\mathbf{1}_{\{\sigma^y\le t\}}(M_{t,1}-M_{t-\sigma^y,1})]\le \sum_{k=1}^{\lfloor L\cdot t\rfloor} ce^{-C_1\frac{k-1}{L}}E[M_{t,1}-M_{t-\frac{k}{L},1}]+ce^{-C_1(t-\eta)}E[M_{t,1}]. \label{eq:FixZwei}
\end{align}
In particular we can handle the cases $y = 1$ and $y = -1$ at once.

Let $j\in\N$ be arbitrary but fixed and $\delta\in(0,1)$. Furthermore, take $x^\ast$ such that $\limsup_{t\to \infty}E[M_{t,1}]/t\le x^\ast$ which exists by Lemma \ref{lem:limsup}. Define $t_0^{(j,\delta,\eta)} := (j-1)/L+\eta$, 
\begin{align*}
t_{k+1}^{(j,\delta,\eta)} &:= \inf\left\{t\in \N_{L,\eta} : t>t_k^{(j,\delta,\eta)}, E[M_{t,1}-M_{t-\frac{j}{L},1}]\le \frac{2}{L\delta}x^\ast\cdot j\cdot e^{\frac{C_1}{2L}(j-1)} \right\}.
\end{align*}
In the following, we prove that this is well defined and that \eqref{eq:FixZwei} is bounded along a suitable subsequence of the sequences $(t_k^{(j,\delta,\eta)})_{k\in\N}$.

We have that $t_{k+1}^{(j,\delta,\eta)}<\infty$, since otherwise we would have that
\begin{align*}
E[M_{t_k^{(j,\delta,\eta)}+\frac{nj}{L}}] &=  \sum_{l=1}^n E[M_{t_k^{(j,\delta,\eta)}+\frac{lj}{L}}-M_{t_k^{(j,\delta,\eta)}+\frac{(l-1)j}{L}}]+E[M_{t_k^{(j,\delta,\eta)}}]\\
&\ge \frac{2}{L\delta}x^\ast nj e^{\frac{C_1}{2L}(j-1)}+E[M_{t_k^{(j,\delta,\eta)}}]
\end{align*}
which would yield $\limsup_{t\to \infty} E[M_t]/t\ge x^\ast \cdot2e^{\frac{C_1}{2L}(j-1)}/\delta> x^\ast$ contradicting the choice of $x^\ast$.  Set 
\begin{align*}
A_n^{(j,\delta,\eta)} := \left\{t\in \N_{L,\eta} : t < n/L+\eta, t\not\in \{t_k^{(j,\delta,\eta)}\}_{k\in\N}\right\},\quad  K_n^{(j,\delta,\eta)} &:= |A_n^{(j,\delta,\eta)}|.
 \end{align*} We want to estimate $K_n^{(j,\delta,\eta)}$. For this define 
\begin{align*}
\widetilde{A}_n^{(j,\delta,\eta)} &:= \left\{t\in \N_{L,\eta} :  t < n/L+\eta, E[M_{t,1}-M_{t-1/L,1}]> \frac{2}{L \delta}x^\ast e^{\frac{C_1}{2L}(j-1)}\right\}, \\\widetilde{K}_n^{(j,\delta,\eta)} &:= |\widetilde{A}_n^{(j,\delta,\eta)}|.
\end{align*} 
The following lemma compares $K_n^{(j,\delta,\eta)}$ with $\widetilde{K}_n^{(j,\delta,\eta)}$. 
\begin{lemma}\label{lem:ineqKn}
We have that
\[K_n^{(j,\delta,\eta)}\le j\widetilde{K}_n^{(j,\delta,\eta)}+j.\]
\end{lemma}
\begin{proof}
Let $t\in \N_{L,\eta}$ with $j/L+\eta\le t< n/L+\eta $ such that $t,t-\frac{1}{L},\dots,t-\frac{j-1}{L}\not\in \widetilde{A}_n^{(j,\delta,\eta)}$. This implies that for $k\in\{0,\dots, j-1\}$
\[
E[M_{t-k/L,1}-M_{t-(k+1)/L,1}]\le \frac{2}{L\delta}x^\ast e^{\frac{C_1}{2L}(j-1)}.
\] Summing these inequalities gives that
\[
E[M_{t,1}-M_{t-j/L,1}]\le j\frac{2}{L\delta}x^\ast e^{\frac{C_1}{2L}(j-1)},
\] i.e. $t\not\in A_n^{(j,\delta,\eta)}$. So for $j/L+\eta\le t\in A_n^{(j,\delta,\eta)}$, there exists a $\varphi(t)\in \left\{t-\frac{j-1}{L},\dots, t\right\}\cap \widetilde{A}_n^{(j,\delta,\eta)}$. If there are multiple elements in the intersection, $\varphi(t)$ is chosen minimal. Also let $\varphi\left(\eta+\frac{1}{L}\right) = \dots =\varphi\left(\eta+\frac{j-1}{L}\right) =: \dagger$, since $\eta+\frac{1}{L},\dots, \eta+\frac{j-1}{L}$ are always in $A_n^{(j,\delta,\eta)}$. This then yields a map
\[
\varphi : A_n^{(j,\delta,\eta)}\to \widetilde{A}_n^{(j,\delta,\eta)}\cup\{\dagger\}, t\mapsto \varphi(t).
\]
We have that $|\varphi^{-1}(t')|\le j$ for all $t'\in \widetilde{A}_n^{(j,\delta,\eta)}\cup\{\dagger\}$ as well as that
\[
A_n^{(j,\delta,\eta)} = \varphi^{-1}(\widetilde{A}_n^{(j,\delta,\eta)}\cup \{\dagger\}) = \bigcup_{t'\in \widetilde{A}_n^{(j,\delta,\eta)}\cup\{\dagger\}}\varphi^{-1}(t')
\]
by definition. This implies that
\[
K_n^{(j,\delta,\eta)} = |A_n^{(j,\delta,\eta)}|\le \sum_{t'\in \widetilde{A}_n^{(j,\delta,\eta)}\cup\{\dagger\}}|\varphi^{-1}(t')|\le |\widetilde{A}_n^{(j,\delta,\eta)}\cup\{\dagger\}|\cdot j = j\widetilde{K}_n^{(j,\delta,\eta)}+j.\qedhere
\]
\end{proof}
\begin{corollary}\label{cor:FixjDenseSub}
We have that
\[
\limsup_{n\to \infty} \frac{K_n^{(j,\delta,\eta)}}{n}\le \frac{j\delta}{2e^{\frac{C_1}{2L}(j-1)}}.
\]
\end{corollary}
\begin{proof}
By Lemma \ref{lem:Monoton} we have that $E[M_{t,1}-M_{t-1/L,1}] \ge 0$ for all $t\ge 1/L$. Together with the definition of $\widetilde{K}_n^{(j,\delta, n)}$ this yields that $
E[M_{\eta+n/L}]> 2(L\delta)^{-1}x^\ast e^{\frac{C_1}{2L}(j-1)} \widetilde{K}_n^{(j,\delta,\eta)}.$ Using Lemma \ref{lem:limsup} this implies that
\[\limsup_{n\to\infty} \frac{\widetilde{K}_n^{(j,\delta,\eta)}}{n} \le \frac{\delta}{2e^{\frac{C_1}{2L}(j-1)}}.\] With Lemma \ref{lem:ineqKn} this yields the statement.
\end{proof}
Now define $A^{(j,\delta,\eta)} := \{t\in \N_{L,\eta} : t\not\in \{t_k^{(j,\delta,\eta)}\}_{k\in\N}\}$ and
\[
B^{\delta,\eta} := \left\{t\in \N_{L,\eta} : t\not\in \bigcup_{j=1}^{\left\lceil \frac{2L}{C_1}\log(t)\right\rceil} A^{(j,\delta,\eta)}\right\}.
\]
In the next two lemmata we  bound $E[\mathbf{1}_{\{\sigma^y\le t\}}(M_{t,1}-M_{t-\sigma^y,1})]$ for $t\in B^{\delta,\eta}$ and examine how dense $B^{\delta,\eta}\subseteq \N_{L,\eta}$ is.
\begin{lemma}\label{lem:bounded}
For $t\in B^{\delta,\eta}$ we have that $
E[\mathbf{1}_{\{\sigma^y\le t\}}(M_{t,1}-M_{t-\sigma^y,1})]\le C,$
with $C$ independent of $t$.
\end{lemma}
\begin{proof}
For $t\in B^{\delta,\eta}$ and all $j\in \left\{1,\dots, \left\lceil \frac{2 L}{C_1}\log(t)\right\rceil\right\}$ we have, by definition, that $E[M_{t,1}-M_{t-j/L,1}]\le 2x^\ast je^{\frac{C_1}{2L}(j-1)}/(\delta L)$. Furthermore, we know that there exists a $c^\ast\ge 0$ such that $E[M_{t,1}-M_{t-j/L,1}]\le E[M_{t,1}] \le c^\ast \cdot t$ for all $t\in \N_{L,\eta}$, since $\limsup_{t\to\infty} E[M_t]/t\le x^\ast$ by Lemma \ref{lem:limsup}. These inequalities as well as \eqref{eq:FixZwei} imply that for $t\in B^{\delta,\eta}$
\begin{align*}
E[\mathbf{1}_{\{\sigma^y\le t\}}(M_{t,1}-M_{t-\sigma^y,1})]&\le \sum_{k=1}^{\lfloor L\cdot t\rfloor} ce^{-C_1\frac{k-1}{L}}[M_{t,1}-M_{t-k/L,1}]+ce^{-C_1(t-\eta)}E[M_{t,1}]\\
&\le \sum_{k=1}^{\lceil \frac{2L}{C_1}\log(t)\rceil} ce^{-C_1\frac{k-1}{L}}\frac{2}{\delta L}x^\ast k e^{\frac{C_1}{2L}(k-1)}\\
&\qquad+\sum_{k=\lceil \frac{2L}{C_1}\log(t)\rceil+1}^{\lfloor L\cdot t\rfloor}  ce^{-C_1\frac{k-1}{L}} c^\ast\cdot t+ce^{-C_1(t-\eta)}c^\ast\cdot t\\
&\le \sum_{k=1}^\infty ce^{-\frac{C_1}{2L}(k-1)}\frac{2}{\delta L}x^\ast k+\hskip-.22cm\sum_{k=\lceil \frac{2L}{C_2}\log(t)\rceil+1}^{\lfloor L\cdot t\rfloor} ce^{-C_1\frac{k-1}{L}}c^\ast e^{\frac{C_1}{2L}k}+c'\\
&\le \widetilde{c}\frac{2x^\ast}{\delta L}+\sum_{k=1}^\infty ce^{-\frac{C_1}{2L}k+\frac{C_1}{L}}c^\ast+c'\\
&\le \widetilde{c}\left(\frac{2x^\ast}{\delta}+c^\ast\right)+c',
\end{align*}
where the exact value of $\widetilde{c}$ changes from line to line and $c'$ is just a constant, which bounds $ce^{-C_1(t-\eta)}c^\ast t$ for all $t\ge 0$. This proves Lemma \ref{lem:bounded}.
\end{proof}
\begin{lemma}\label{lem:dense}
Let $(t_k^{\delta,\eta})_{k\in\N}$ be a monotonically increasing enumeration of $B^{\delta,\eta}$ and  $\delta e^{\frac{C_1}{L}}/(e^\frac{C_1}{2L}-1)^2\in (0,1)$. Then, one has that
\[
\limsup_{n\to \infty} \frac{t_n^{\delta,\eta}}{n}\le \frac{1}{L}+\delta \frac{e^{\frac{C_1}{L}}}{L(e^\frac{C_1}{2L}-1)^2}.
\]
\end{lemma}
\begin{proof}
Consider $A_n^{\delta,\eta} := \left\{t\in \N_{L,\eta} : t<n/L+\eta, t\not\in B^{\delta,\eta}\right\}$, $K_n^{\delta,\eta} := |A_n^{\delta,\eta}|$. Then
\[
A_n^{\delta,\eta}\subseteq  \bigcup_{j=1}^{\left\lceil \frac{2 L}{C_1}\log(n)\right\rceil} A_n^{(j,\delta,\eta)}
\]
which implies that $
K_n^{\delta,\eta}\le \sum_{j=1}^{\left\lceil \frac{2L}{C_1}\log(n)\right\rceil} |A_n^{(j,\delta,\eta)}|.$
This gives that
\[
\frac{K_n^{\delta,\eta}}{n}\le \sum_{j=1}^{\left\lceil \frac{2L}{C_1}\log(n)\right\rceil} \frac{K_n^{(j,\delta,\eta)}}{n} = \sum_{j=1}^\infty \mathbf{1}_{\left\{j\le \left\lceil \frac{2L}{C_1}\log(n)\right\rceil\right\}} \frac{K_n^{(j,\delta,\eta)}}{n}.
\]
We now want to apply Fatou's lemma and for this need to bound the summands for constant $j$. Thus let $n,j\in\N$. We know, by Lemma \ref{lem:ineqKn}, that $K_n^{(j,\delta,\eta)}\le j\widetilde{K}_n^{(j,\delta,\eta)}+j$ and, by the calculation in Corollary \ref{cor:FixjDenseSub}, that
\[\widetilde{K}_n^{(j,\delta,\eta)}\le E[M_{\frac{n}{L}+\eta}]\cdot \frac{L\delta}{2 x^\ast e^{\frac{C_1}{2L}(j-1)}}.\]
By Lemma \ref{lem:limsup} there is a $c^\ast$ such that $E[M_t]\le c^\ast (t-\eta)$ for all $t\in\N_{L,\eta}$, this implies that
\begin{align*}
\frac{K_n^{(j,\delta,\eta)}}{n}\le \frac{jc^\ast \delta}{2x^\ast e^{\frac{C_1}{2L}(j-1)}}+\frac{j}{n}.
\end{align*}
This implies that
\[
\mathbf{1}_{\left\{j\le \left\lceil \frac{2L}{C_1}\log(n)\right\rceil\right\}} \frac{K_n^{(j,\delta,\eta)}}{n}\le \frac{jc^\ast\delta}{2x^\ast e^{\frac{C_1}{2L}(j-1)}}+\frac{j}{e^{\frac{C_1}{2L}(j-1)}}.
\]
 Since the right hand side is summable, this implies, by Fatou's Lemma, that
\begin{align*}
\limsup_{n\to\infty} \frac{K_n^{\delta,\eta}}{n}&\le \sum_{j=1}^\infty \limsup_{n\to \infty} \mathbf{1}_{\left\{j\le \left\lceil \frac{2L}{C_1}\log(n)\right\rceil\right\}}\frac{K_n^{(j,\delta,\eta)}}{n}\\
&\le \sum_{j=1}^\infty \frac{j\delta}{2e^{\frac{C_1}{2L}(j-1)}}\\
&= \frac{\delta}{2}\cdot \frac{e^{\frac{C_1}{L}}}{(e^\frac{C_1}{2L}-1)^2},
\end{align*}
where the second to last inequality uses Corollary \ref{cor:FixjDenseSub}.

This implies that
\[
\liminf_{n\to \infty} \frac{|\{t\in \N_{L,\eta}: t<n/L+\eta, t\in B^{\delta,\eta}\}|}{n} \ge \left(1-\frac{\delta}{2}\cdot \frac{e^{\frac{C_1}{L}}}{(e^\frac{C_1}{2L}-1)^2}\right) =: c_\delta.
\]
This in turn implies for $\delta e^{\frac{C_1}{L}}/(e^\frac{C_1}{2L}-1)^2\in (0,1)$ that $
\limsup_{n\to \infty} t^{\delta,\eta}_{\left\lceil c_\delta\cdot n\right\rceil}n/L)\le 1 $
and thus we have that
\begin{equation*}
\limsup_{n\to \infty} \frac{t^{\delta,\eta}_n}{n}\le L^{-1}c_\delta^{-1}\le  \frac{1}{L}+\delta \frac{e^{\frac{C_1}{L}}}{L(e^\frac{C_1}{2L}-1)^2}.\qedhere
\end{equation*}
\end{proof}
Summed up we have
\begin{corollary}\label{cor:Unten}
Fix $\delta>0$ and $\eta\in [0,1/L)$. Then, there exists a deterministic subsequence $(s_k^{\delta,\eta})_{k\in\N}$ of $\N_{L,\eta}$ such that $(E[\mathbf{1}_{\{\tau\neq \tau_s\}}(M_{s_k^\delta,1}-M_{s_k^\delta+\tau})])_{k\in\N}$ is bounded and $\limsup_{k\to \infty} s_k^{\delta,\eta}/k\le (1+\delta)/L$.
\end{corollary}
\begin{proof}
Let $\widetilde{\delta} := \delta(e^{\frac{C_1}{2L}}-1)^2/e^{\frac{C_1}{L}}$. Consider $(s_k^{\delta,\eta})_{k\in\N}$ an increasing enumeration of $B^{\widetilde{\delta},\eta}$. By Lemma \ref{lem:dense} we have that
\[
\limsup_{k\to \infty} \frac{s_k^{\widetilde{\delta},\eta}}{k} \le L^{-1}\left(1+\widetilde{\delta}\frac{e^{\frac{C_1}{L}}}{\left(e^{\frac{C_1}{2L}}-1\right)^2}\right) = \frac{1+\delta}{L}.
\]
By equations \eqref{eq:Splitwrtfjump}, \eqref{eq:Decomp} and \eqref{eq:FixZwei} and Lemmata \ref{lem:BoundSideterm} and \ref{lem:bounded} as well as $s_k^{\delta,\eta}\in \N_{L,\eta}$ for all $k$ we know that $(E[\mathbf{1}_{\{\tau\neq \tau_s\}}(M_{s_k^\delta,1}-M_{s_k^\delta+\tau})])_{k\in\N}$ is bounded.
\end{proof}
\subsection{Proof of Theorem \ref{th:Main}}
Let $(t_{k}^{\frac{\delta}{2},\eta})_{k\in\N}$ be a subsequence according to Lemma \ref{lem:Oben} and $(s_{k}^{\frac{\delta}{2},\eta})_{k\in\N}$ a subsequence according to Corollary \ref{cor:Unten}. Now, consider
\[
A^{\delta,\eta} :=  \left\{t\in \N_{L,\eta} : t\in \left\{t_k^{\frac{\delta}{2},\eta}\right\}_{k\in\N}\cap\left\{s_k^{\frac{\delta}{2},\eta}\right\}_{k\in\N}\right\}
\]
and let $K_n^\delta := |\{t\in \N_{L,\eta}: t<n/L+\eta, t\not\in A^{\delta,\eta}\}|$. We have that
\begin{align*}
K_n^{\delta,\eta} &\le \left|\left\{t\in  \N_{L,\eta}:  t< n/L+\eta, t\not\in\left\{t_k^{\frac{\delta}{2},\eta}\right\}\right\}\right|+\left|\left\{t \in \N_{L,\eta} : t< n/L,t\not\in \left\{s_k^{\frac{\delta}{2},\eta}\right\} \right\}\right| \\
&=: K_{n,1}^\delta+K_{n,2}^\delta.
\end{align*}
By the construction of the sequences $(t_{k}^{\delta/2,\eta})_{k\in\N}$, $(s_k^{\delta/2,\eta})_{k\in\N}$ we know that for $j\in\{1,2\}$ $\limsup_{n\to \infty}K_{n,j}^{\delta,\eta}/n\le \delta/4$, $j\in\{1,2\}$. This implies that $\limsup_{n\to \infty} K_n^{\delta,\eta}/n\le \delta/2$. As in Lemma \ref{lem:Oben} and Corollary \ref{cor:Unten} this implies that for $(t_k^{\delta,\eta})_{k\in\N}$ an increasing enumeration of $A^{\delta,\eta}$ we have that $
\limsup_{k\to \infty} t_k^{\delta,\eta}/k \le (1+\delta)/L. $

Furthermore, since $t_k^{\delta,\eta} \in \{t_k^{\delta/2,\eta}\}_{k\in\N}\cap\{s_k^{\delta/2,\eta}\}_{k\in\N}$ we have by Lemma \ref{lem:Oben} and  Corollary \ref{cor:Unten} as well as \eqref{eq:RewriteEins} that for all $k\in\N$
\begin{align*}
E[|M_{t_k^{\delta,\eta},1}-M_{t_k^{\delta,\eta},2}|]&\le c_{\mathrm{ei},L}^{-1}(E[M_{t^{\delta,\eta}+\frac{1}{L}}-M_{t^{\delta,\eta}}]+E[\mathbf{1}_{\{\tau\neq \tau_s\}}(M_{t^{\delta,\eta},1}-M_{t^{\delta,\eta}+\tau})])\le C
\end{align*}
with $C$ independent of $k$. This implies that $(M_{t_k^{\delta,\eta}}-E^\xi[M_{t_k^{\delta,\eta}}])_{k\in\N}$ is tight with respect to the annealed measure.\qed
\section{Quenched tightness-Proof of Theorem \ref{Theo:Quenched}}
We now sketch which changes to the argument are necessary to get Theorem \ref{Theo:Quenched}. 
The analogue to inequality \eqref{eq:MainInequality} is
\begin{align*}
E^\xi[|M_{t,1}-M_{t,2}|]\le c_{\xi(0),L}^{-1}\left(E^\xi[M_{t+1/L}-M_t]+E^\xi[\mathbf{1}_{\{\tau\neq \tau_s\}}(M_{t,1}-M_{t+\tau}]\right).
\end{align*}
Since Lemma \ref{lem:limsup} is proven by comparing $E^\xi$ with $E^{\mathrm{es}}$ we can derive that there is a $x^\ast\in\R$ with $
\limsup_{t\to \infty} \frac{E^{\xi}[M_t]}{t}\le x^\ast$
for $\mathbb{P}$-a.e. $\xi$. This allows us to replace $E$ by $E^\xi$ in Lemma  \ref{lem:Oben} and no further changes are needed. Note that $t_k^{\delta,\eta}$ will be $\xi$ dependent, since the condition $E^\xi[M_{t+1/L}-M_t]\le \frac{2x^\ast}{L\delta}$ depends on $\xi$.

In section 2.3, replacing $E$ and $P$ by $E^\xi$ and $P^\xi$ everywhere suffices and Lemma \ref{lem:sigExpTail} is already stated for $\mathbb{P}$-a.e. $\xi$. Again the condition $E^\xi[M_{t,1}-M_{t-j/L,1}]\le \frac{2}{L\delta}x^\ast j e^{\frac{C_1}{2L}(j-1)}$ being $\xi$ dependent forces the $s_k^{\delta,\eta}$ in Corollary \ref{cor:Unten} to be $\xi$ dependent. 

Combining these ingredients to prove Theorem \ref{Theo:Quenched} is parallel to subsection 2.4, one only needs to replace $E$ by $E^\xi$ in the last display.

We have not managed to find a deterministic subsequence $(t_k)_{k\in\N}$ of $\N_{L,\eta}$ such that $(M_{t_k}-E^\xi[M_{t_k}])_{k\in\N}$ is tight for $\mathbb{P}$-a.e. $\xi$.

\printbibliography
\end{document}